\documentclass[12pt]{amsart}
\usepackage{amssymb}
\usepackage{amsfonts}
\usepackage{graphicx}
\usepackage{epsfig}
\usepackage{amsmath}

\theoremstyle{plain}
\newtheorem{theorem}{Theorem}[section]
\newtheorem{proposition}[theorem]{Proposition}
\newtheorem{lemma}[theorem]{Lemma}

\theoremstyle{definition}

\newtheorem{definition}[theorem]{Definition}
\newtheorem{remark}[theorem]{Remark}

\newcommand{\X}{\mathbb{X}}
\newcommand{\R}{\mathbb{R}}
\newcommand{\C}{\mathbb{C}}
\newcommand{\bc}{\mathbb{C}}
\newcommand{\ra}{{\rightarrow}}

\let\cal\mathcal

\def\<{\langle}
\def\>{\rangle}
\def\ch#1{{{\bf H}^{#1}_{\C}}}

\begin{document}
\title
        {McShane's identity in rank one symmetric spaces}
        \author{Inkang Kim, Joonhyung Kim and Ser Peow Tan}
        \address{Inkang Kim \\School of Mathematics\\
     KIAS\\
     Seoul, 130-722, Korea}
\email{inkang\char`\@ kias.re.kr}%
\address{Joonhyung Kim \\
    Department of Mathematics, Konkuk University\\
           1 Hwayang-dong, Gwangjin-gu\\
           Seoul 143-701, Republic of Korea}
\email{calvary\char`\@snu.ac.kr}%
\address{Ser Peow Tan\\
    Department of Mathematics, National University of
           Singapore\\
           Block S17, Lower Kent Ridge Road\\
           Singapore 119076, Republic of Singapore}
\email{mattansp\char`\@nus.edu.sg}
        \date{}
        \maketitle

\begin{abstract}
In this paper we study McShane's identity in real and complex
hyperbolic spaces and obtain various generalizations of the identity
for representations of surface groups into the isometry groups of rank one symmetric spaces.  Our methods unify most
of the existing methods used in the existing literature for proving this class of identities.
\end{abstract}
\footnotetext[3]{\sl{The first author was partially supported by NRF
grant 2010-0024171.}} \footnotetext[4]{\sl{The second author was supported by NRF
grant 2010-0001194.}} \footnotetext[5]{\sl{The third author was
partially supported by the National University of Singapore academic
research grant R-146-000-115-112.}}

\section{Introduction}
Greg McShane proved a striking identity for the lengths of simple closed geodesics for once-punctured hyperbolic tori in his PhD thesis \cite{mcshane1991thesis} which he generalized later to cusped hyperbolic surfaces in \cite{mcshane1998im}.
A remarkable feature of this identity is that it involved the functions of the lengths of simple closed geodesics on the surface as opposed to the Selberg trace formula, which involved functions of the lengths of all primitive geodesics on the surface.
Since then, there have been many generalizations of the original McShane's
identity, and also alternative proofs; for example by Bowditch  \cite{bowditch1996blms}, \cite{bowditch1997t}, Akiyoshi-Miyachi-Sakuma  \cite{AMS}, Mirzakhani \cite{Mir} and Tan-Wong-Zhang \cite{tan-wong-zhang2004cone-surfaces}, \cite{tan-wong-zhang2004schottky}, \cite{tan-wong-zhang2005gmm}.

The identities for hyperbolic surfaces with boundaries and/or cone points   have found important applications by Mirzakhani  in the computation of the Weil-Petersson volume of the moduli space of hyperbolic surfaces with boundaries and/or cone singularities (\cite{Mir}, \cite{tan-wong-zhang2004cone-surfaces}), developed further by her in \cite{Mir2} to give an alternative proof of the Kontsevich-Witten theorem on the Virasoro relations for the  intersection numbers of tautological classes on the moduli space of stable curves as well as in \cite{Mir3} to study the asymptotic growth of the number of simple closed geodesics on hyperbolic surfaces. Indeed, the identities  have, together with the Fenchel-Nielsen coordinates proved to be important tools in the study of the Weil-Petersson geometry of the moduli space of curves.

More recently, Labourie-McShane formulated the identity in terms of
cross-ratios \cite{LM}, thereby generalizing the identities to representations of surface groups into ${\rm PSL}(n, \mathbb R)$. In this article, we use their idea to give the most
general version of the identity for surfaces with boundary in rank one symmetric spaces in terms of the cross-ratio.
More specifically, when the symmetric space is complex hyperbolic, we have
$\mathbb C$-valued cross-ratios and the McShane's identity is a
complex equation. For 3-dimensional real hyperbolic case, the
cross-ratio is also $\mathbb C$-valued and we also recover some of
 known generalizations  by Bowditch \cite{bowditch1996blms}, and
Akiyoshi-Miyachi-Sakuma \cite{AMS}.

Our main results in this paper are identities for hyperbolic mapping tori and
complex hyperbolic manifolds using cross-ratio techniques, see Theorems \ref{main} and \ref{main2}.
Recently, there has been much work on the representation spaces of surface groups into other Lie groups besides ${\rm PSL}(2,\mathbb{R})$ and ${\rm PSL}(2, \mathbb{C})$ and we expect that the generalized identities will eventually also play an important role in the understanding of the geometry of these spaces.

\section{Preliminaries}
Let $R$ be one of the four division rings. Due to the lack of
commutativity for quaternions and octonions, we will not define the
quaternion valued or octonion  valued cross-ratios.
There is a way to define exponential (hence logarithm) on
quaternions, see for example \cite{tan-wong-zhang2011Delambre} but it does not satisfy the usual rules like
$e^xe^y=e^{x+y}$, so it is not useful to us. For those two division
rings, we will define only real valued cross-ratios.

Let $X$ be a rank one symmetric space in the unit ball model. For
$x,y \in X$, the distance between these two points are:
$$\cosh(d(x,y))=\frac{(|1-\langle x,y \rangle |^2+2R\langle x,y\rangle)^{1/2}}
{(1-\langle x,x\rangle)^{1/2}(1-\langle y,y\rangle)^{1/2}}$$ where
$R<v,w>=Re(v_1\bar{v_2})(w_2\bar{w_1})-Re(\bar{v_2}w_2)(\bar{w_1}v_1)$
for the Cayley hyperbolic case and $R<v,w>=0$ for other cases. See
\cite{Mo}.

We define the real cross ratio in the unit ball model by:
$$[x,z,y,w]_\R=\frac{\langle \langle z,x \rangle \rangle \langle \langle w,y
\rangle \rangle}{ \langle \langle w,x \rangle \rangle \langle
\langle z,y \rangle \rangle}$$ where $\langle \langle x,y \rangle
\rangle={(|1-\langle x,y \rangle |^2+ 2R\langle x,y
\rangle)^{1/2}}$.

We can rewrite the definition in terms of the generalized Heisenberg
group $N$ of $X$. Note that the ideal boundary of $X$ is a one point
compactification of $N$ which is a two step Nilpotent group. The
following is proved in \cite{Kim} (note that the order of entries in
the definition of the cross-ratio is changed for our convenience).
\begin{proposition}The real cross ratio above can be written as:
$$[g_1,g_3,g_2,g_4]_\R=\frac{|g_3^{-1}g_1|^2|g_4^{-1}g_2|^2}
{|g_4^{-1}g_1|^2|g_3^{-1}g_2|^2}$$ where $x,y,z,w$ correspond to
$g_1,g_2,g_3,g_4$ in the Heisenberg group respectively.
\end{proposition}
The ideal boundary of rank one symmetric space has a so-called
Gromov visual metric and it is equivalent to the left-invariant
Heisenberg metric $d(g, h)=|h^{-1}g|$. We will use this metric when
we talk about the H\"older structure. Note that in real hyperbolic
case, the associated Heisenberg group is just $\R^n$ and the real
cross-ratio becomes
$$[X,Y,Z,W]_\R=\frac{|X-Y||Z-W|}{|X-W||Z-Y|}, \quad X,Y,Z,W\in \R^n.$$
Specifically, for real 2 or 3-dimensional case, one can define a real (
resp. complex) valued cross-ratio as
$[x,y,z,w]=\frac{(x-y)(z-w)}{(x-w)(z-y)}$ for $x,y,z,w\in \bc$.

 Now we define the
$\mathbb C$-valued cross-ratio on the ideal boundary of complex
hyperbolic manifold.
 If $a,b,c,d$ are
represented by $z_1,z_2,z_3,z_4$ on the boundary of unit ball model,
and $\tilde z_1=(z_1,1),\tilde z_2=(z_2,1),\tilde z_3=(z_3,1),\tilde
z_4=(z_4,1)$ are lifts in paraboloid model, then
$$[a,b,c,d]=\frac{\langle\tilde z_1,\tilde z_2\rangle\langle\tilde z_3,\tilde z_4\rangle}
{\langle\tilde z_1,\tilde z_4\rangle\langle\tilde z_3,\tilde
z_2\rangle}$$ where
$\langle(x_1,\cdots,x_n),(y_1,\cdots,y_n)\rangle=\sum_{i=1}^{n-1}
x_i\bar y_i-x_n\bar y_n$. Since $z_i$ are on the boundary of the
unit ball model, their norm is 1 and so $\langle\tilde z_i,\tilde
z_i\rangle=0$. The cross-ratio above satisfies:
\begin{enumerate}
\item  $[a,b,c,d]=0$ iff $a=b$ or $c=d$.
\item  $[a,b,c,d]=1$ iff $a=c$ or $b=d$.
\item  $[x,y,z,t]=[x,y,w,t][w,y,z,t]$
\item $[x,y,z,t]=[x,y,z,w][x,w,z,t]$
\item  $[x,y,z,t]=[z,t,x,y]$
\item  $ [x,y,z,t]=[z,y,x,t]^{-1}$
\item  $[x,y,z,t]=[x,t,z,y]^{-1}$
\item  $[x,y,z,t]_\R=|[x,y,z,t]|$
\end{enumerate}

For a hyperbolic isometry $\gamma$, the period of $\gamma$ is
defined as
$$\ell(\gamma)=\log [\gamma^-,\gamma y,\gamma^+, y].$$
In rank one symmetric space, it is shown in \cite{Kim} that
$$\text{Re}\log [\gamma^-,\gamma y,\gamma^+, y]=\log |[\gamma^-,\gamma y,\gamma^+, y]|=
l(\gamma)$$ where $l(\gamma)$ is a translation length of $\gamma$.
Henceforth, the cross-ratio is understood as $\bc$-valued for
complex hyperbolic space and 3-dimensional real hyperbolic space,
and $\R$-valued otherwise.

On complex hyperbolic space $\ch{n}$, a general reference is
\cite{Go}, but \cite{PP} is enough to follow this paper. If we
choose a second Hermitian form on $\C^{n,1}$, for column vectors $z=(z_1, \cdot\cdot\cdot, z_{n+1})$ and  $w=(w_1, \cdot\cdot\cdot, w_{n+1})$, $\displaystyle{\<z,w\>=w^*Jz=z_1
\overline{w_{n+1}}+z_{n+1} \overline{w_1}+\sum_{j=2}^n z_j \overline{w_j}}$, where
$$
J=\left[\begin{matrix} 0 & 0 & 1 \\ 0 & I_{n-1} & 0
\\ 1 & 0 & 0
\end{matrix}\right],
$$
and $w^*$ is the Hermitian transpose of $w$.\\
 We define the Siegel domain $\mathfrak{S}$ of a complex
hyperbolic space $\ch{n}$ by identifying points of $\mathfrak{S}$
with their horospherical coordinates, $z=(\zeta,v,u) \in \C^{n-1}
\times \R \times \R_+$. The boundary of $\mathfrak{S}$ is given by
$H_0\cup\{q_{\infty}\}$, where $q_{\infty}$ is a distinguished
point at infinity and $H_0=\C^{n-1} \times \R \times\{0\}$.\\
Define a map $\psi:\overline{\mathfrak{S}}\rightarrow \mathbb{P} \C^{n,1}$
by
$$
\psi:(\zeta,v,u)\mapsto \left[\begin{matrix}
(-\mid\zeta\mid^2-u+iv)/2
\\ \zeta \\ 1\end{matrix}\right] \hbox{ for }
(\zeta,v,u)\in\overline{\mathfrak{S}}-\{q_{\infty}\} \hbox{   ; }
\psi:q_{\infty}\mapsto\left[\begin{matrix} 1
\\ 0 \\ 0\end{matrix}\right].
$$
Then $\psi$ maps $\mathfrak{S}$ homeomorphically to the set of
points $z$ in $\mathbb{P} \C^{n,1}$ with $\langle z,z\rangle<0$, and maps
$\partial \mathfrak{S}$ homeomorphically to the set of points $z$
in $\mathbb{P} \C^{n,1}$ with $\langle z,z\rangle=0$. We write $\psi(\tilde{z})=z$.\\
The Bergman metric on $\mathfrak{S}$ is given by the
distance formula
$$
\cosh^2\left(\frac{\rho(\tilde{z},\tilde{w})}{2}\right)=\frac{\<z,w\>\<w,z\>}{\<z,z\>\<w,w\>}.
$$
The boundary of a complex hyperbolic space is identified with  the
one point compactification of the Heisenberg group. The
$(2n-1)$-dimensional
Heisenberg group $\mathfrak{N}$ is $\C^{n-1} \times \R$ with the group
law
$$
(\zeta_1,v_1)\diamond(\zeta_2,v_2)=(\zeta_1+\zeta_2,v_1+v_2+2Im\langle\langle\zeta_1,\zeta_2\rangle\rangle)
$$
where $\langle\langle\zeta_1,\zeta_2\rangle\rangle$ is the
standard positive-definite Hermitian form on $\C^{n-1}$.\\
The Heisenberg norm assigns to $q=(\zeta,v)$ in
$\mathfrak{N}$ the nonnegative real number
$$ \mid q \mid =(\mid \zeta \mid ^4 + v^2)^{1/4}.
$$
This gives rise to the Cygan metric on $\mathfrak{N}$ by
\begin{align*}
\rho_0((\zeta_1,v_1),(\zeta_2,v_2)) &= \mid (\zeta_1,v_1)^{-1} \diamond (\zeta_2,v_2) \mid \\
&= \mid \mid \zeta_1-\zeta_2 \mid^2-iv_1+iv_2-2iIm\langle\langle\zeta_1,\zeta_2\rangle\rangle \mid^{1/2}.
\end{align*}

The full group of holomorphic isometries of $\ch{2}$ is $PU(2,1)$.
Since $PU(2,1)=SU(2,1)/ \{I, \omega I, \omega^2 I \}$, where
$\omega$ is a non-real cube root of unity, we may consider
$SU(2,1)$ instead of $PU(2,1)$. Then, as in the real hyperbolic
isometry, we can classify isometries of complex hyperbolic space
as loxodromic, parabolic, or elliptic by their
fixed points.\\
If $A \in SU(2,1)$ is a matrix representing a loxodromic isometry
and $e^\lambda$ is an attracting eigenvalue of $A$,
$$
A=\left[\begin{matrix} e^\lambda & 0 & 0 \\ 0 & e^{\overline{\lambda}-\lambda} & 0 \\
0 & 0 & e^{-\overline{\lambda}}
\end{matrix}\right],
$$
where $\lambda \in S=\{\lambda \in \C \mid Re(\lambda)>0,
Im(\lambda) \in (-\pi, \pi] \}$.\\
In \cite{PP}, Parker and Platis introduced the cross-ratio
variety. Given distinct points $z_1, z_2, z_3, z_4 \in
\partial\ch{2}$, they defined

$$
\X_1=[z_4, z_2, z_3, z_1], \quad \X_2=[z_4, z_3, z_2, z_1],  \quad
\X_3=[z_4, z_3, z_1, z_2].
$$
Then, the following identities hold.

\begin{proposition}(Proposition 5.2 in \cite{PP})
Let $\X_1$, $\X_2$, and $\X_3$ be defined as above. Then,
$$
|\X_2|^2=|\X_1||\X_3|,
$$
$$
2|\X_1|^2Re(\X_3)=|\X_1|^2+|\X_2|^2+1-2Re(\X_1+\X_2).
$$
\end{proposition}

An important property on complex cross-ratio worth taking note of for our purposes is the following.

\begin{lemma}
The complex valued cross-ratio is H\"older if we fix three points $x,z,t$ and $y$ varies away from $x$ and $z$,
where $x$ and $z$ are not close to $y$.
\end{lemma}
\begin{proof}
Since $SU(2,1)$ acts on $2$-transitively on $\ch{2}$, we may assume that $x=0$ and $z=\infty$.
Let $y_1=(\zeta_1,v_1), y_2=(\zeta_2,v_2)$ and $t=(\zeta,v)$. Then as elements of $\mathbb{P} \C^{2,1}$,
$$
x=\left[\begin{matrix} 0 \\ 0 \\ 1 \end{matrix}\right],
z=\left[\begin{matrix} 1 \\ 0 \\ 0
\end{matrix}\right], y_1=\left[\begin{matrix} -\mid \zeta_1 \mid^2+iv_1 \\ \sqrt{2}\zeta_1 \\ 1
\end{matrix}\right], y_2=\left[\begin{matrix} -\mid \zeta_2 \mid^2+iv_2 \\ \sqrt{2}\zeta_2 \\ 1
\end{matrix}\right], t=\left[\begin{matrix} -\mid \zeta \mid^2+iv \\ \sqrt{2}\zeta \\ 1
\end{matrix}\right].
$$
Then $\displaystyle{\<z,t\>=\<z,y_1\>=\<z,y_2\>=1}$ and, for the convenience of calculation, let $\displaystyle{\zeta_1=r_1 e^{i\theta_1}, \zeta_2=r_2 e^{i\theta_2}, d:=v_2-v_1, \theta:=\theta_2-\theta_1}$ and\\
$\displaystyle{K:=\mid\<x,t\>\mid=\mid -\mid\zeta\mid^2-iv \mid}$.
Since $y_1$ and $y_2$ are not close to $x$ and $z$, there exist positive constants $a$ and $b$ such that $a<\sqrt{r_{i}^{4}+v_{i}^2}<b$, for $i=1,2$. Then, obviously, $r_{i}^2, \mid v_i \mid <b$.\\
Let $\rho:=\rho_0(y_1,y_2)$. By easy calculation,
\begin{align*}
\rho^4 &= \mid r_2^2+r_1^2-2r_1r_2e^{i\theta}-id \mid^2 \\
&= (r_2^2+r_1^2-2r_1r_2\cos\theta)^2+(d+2r_1r_2\sin\theta)^2.
\end{align*}
Hence,
$$
\sqrt{(r_1^2+r_2^2)^2+d^2}-2r_1r_2 \leq \rho^2 \leq \sqrt{(r_1^2+r_2^2)^2+d^2}+2r_1r_2.
$$
Furthermore,
\begin{align*}
A &:=\mid [x,y_1,z,t]-[x,y_2,z,t] \mid\\
&= \mid \frac{\<x,y_1\>\<z,t\>}{\<x,t\>\<z,y_1\>}-\frac{\<x,y_2\>\<z,t\>}{\<x,t\>\<z,y_2\>} \mid\\
&= \frac{1}{K}\mid \<x,y_1\>-\<x,y_2\> \mid\\
&= \frac{1}{K}\mid -r_1^2-iv_1+r_2^2+iv_2 \mid\\
&= \frac{1}{K}\sqrt{(r_2^2-r_1^2)^2+d^2}.
\end{align*}
Then, $\displaystyle{\rho^2 \geq \sqrt{KA^2+4r_1^2r_2^2}-2r_1r_2}$, so $\displaystyle{\rho^4+4r_1^2r_2^2\rho^2 \geq KA^2}$. Therefore,
\begin{align*}
A &\leq \frac{\rho}{\sqrt{K}}(\rho^2+4r_1^2r_2^2)^{1/2}\\
&\leq \frac{\rho}{\sqrt{K}}(\sqrt{(r_1^2+r_2^2)^2+d^2}+2r_1r_2+4r_1^2r_2^2)^{1/2}\\
&\leq \frac{\rho}{\sqrt{K}}(\sqrt{r_1^4+v_1^2+r_2^4+v_2^2+2r_1^2r_2^2+2\mid v_1v_2 \mid}+2r_1r_2+4r_1^2r_2^2)^{1/2}\\
&\leq \frac{\rho}{\sqrt{K}}(\sqrt{8b^2}+2b+4b^2)^{1/2}\\
&= C\rho,
\end{align*}
where $\displaystyle{C=(\frac{\sqrt{8}b+2b+4b^2}{K})}$.
\end{proof}

\begin{remark}
It is known that the real valued cross-ratio is H\"older. See \cite{LM}.
\end{remark}

\section{Finite type hyperbolic surfaces}
Let $X$ be a Riemann surface with geodesic boundaries and punctures
whose Euler number is negative. In this case, there is
$\Gamma\subset PSL(2,\R)$ so that $H^2_\R/\Gamma$ is an infinite
volume hyperbolic surface so that if we truncate all the flaring
ends along the geodesics we obtain $X$. If we give another such
metric, these two hyperbolic metrics are quasi-isometric and there
is a $\pi_1(X)$-equivariant H\"older homeomorphism between two limit
sets. Usually $\partial_\infty \pi_1(X)$ is identified with one of
this limit set with inherited H\"older structure. Note that the
universal cover $\tilde X$ is exactly equal to the convex hull of
the limit set and the ideal boundary of $\tilde X$ is exactly
$\partial_\infty \pi_1(X)$. Any element in $\pi_1(X)$ acts on
$\partial_\infty \pi_1(X)$ as H\"older homeomorphisms.

Let $\Gamma\subset PSL(2,\R)\subset G$ be a natural inclusion
stabilizing totally geodesic Lagrangian $H^2_\R$ in the rank one
symmetric space $X$ associated to $G$. Then it is known that there
are type preserving small deformations  $\rho_t$ around $\Gamma$,
i.e.,  deformations not stabilizing $H^2_\R$ and whose convex core
has a finite volume. For example, a quasifuchsian group in
3-dimensional hyperbolic case. In this case the limit set is still a
Jordan curve in $\partial X$ so that the induced map $f_t$ from the
limit set of $\Gamma$ to the one of $\rho_t$ is a H\"older
homeomorphism \cite{B}. If we identify $\partial_\infty \pi_1(X)$
with the limit set of $\Gamma$, then there is a cross-ratio defined
as
$$B_{\rho_t}(x,y,z,w)=[f_tx,f_ty,f_tz,f_t w]$$ where the right side
 cross-ratio is defined on $\partial X$. Note that if we fix 3 points,
say $x, z, w$, then the function
$B(y)=B_{\rho_t}(x,y,z,w):\partial_\infty \pi_1(X)\ra \R,\bc$ is
H\"older when $y$ varies on some domain away from $x$ and $z$ since
it is a composition of two H\"older maps $f_t$ and
$[f_tx,\cdot,f_tz,f_tw]$.

Let $\alpha$ denote a geodesic boundary of $X$. A pair of pants
embedded in $X$ with first boundary $\alpha$ is an isotopy class
of embeddings of a pair of pants $P$ whose boundary is mapped to
$(\alpha,\beta,\gamma)$ so that $\alpha\gamma\beta=1$ in
$\pi_1(X)$. Note here that the isotopy is allowed to rotate along
$\alpha$, so $(\alpha,\beta,\gamma)$ and
$(\alpha,\alpha^n\beta\alpha^{-n},\alpha^n\gamma\alpha^{-n})$
represent the same isotopy class.

For any two points in $\partial_\infty\pi_1(X)$, one can connect
them by a geodesic on $\tilde X$ and project it down to $X$. If this
geodesic is simple, we say the pair belongs to the Birman-Series
set. Birman and Series showed that this set has zero Hausdorff
dimension.  Suppose $\alpha$ is either boundary or cusp of $X$ and
denote $\alpha^+$ attracting point, $\alpha^-$ repelling point of
$\alpha$ respectively. We give a canonical counterclockwise
orientation on $\partial_\infty\pi_1(X)$. $K_\alpha$ is a subset of
$\partial_\infty\pi_1(X)\setminus \{\alpha^+,\alpha^- \}$ consisting
of $t$ so that $\{\alpha^+,t\}$ is a simple geodesic spiraling
around $\alpha$. If $\alpha$ is a cusp, a geodesic is spiraling
around $\alpha$ simply means that its end is shooting up to the cusp
$\alpha$. $K_\alpha^*$ is the subset of $K_\alpha$, excluding the
$\pi_1(X)$ orbit of $\alpha^+$.

\begin{definition}For $x,y\in K_\alpha^*$, a gap $(x,y)=\{t\in
\partial_\infty\pi_1(X)| x<t<y \ \text{in}\
\text{counterclockwise}\ \text{ordering}\}$ is such that $(x,y)\cap
K_\alpha^*=\emptyset$.
\end{definition}
So $\partial_\infty\pi_1(X)\setminus K_\alpha^*$ is a disjoint union
of gaps. It is known by \cite{mcshane1998im,Mir}
\begin{theorem} For a pair of pants $(\alpha,\beta,\gamma)$,
$(\beta^+,\gamma^-)$ is a gap and if $\beta$ is a boundary element,
then $(\beta^-,\beta^+)$ is a gap. Conversely every gap arises this
way.
\end{theorem}

Associated to these two types of gaps, one can associate gap
functions
\begin{enumerate}
\item  If $P=(\alpha,\beta,\gamma)$ represents a pair of pants with first
boundary $\alpha$, then the gap function is
$$G(P)=\log[\alpha^+,\gamma^-,\alpha^-,\beta^+]$$

\item If $P=(\alpha,\beta,\gamma)$ represents a pair of pants so
that $\beta$ is also a boundary, then
$$G^r(P)=\log[\alpha^+,\beta^+,\alpha^-,\beta^-].$$
\end{enumerate}
Note that by
$$\frac{[\alpha^+,\beta^+,\alpha^-,\zeta]}{[\alpha^+,\beta^-,\alpha^-,\zeta]}=[\alpha^+,\beta^+,\alpha^-,\beta^-]$$ for any $\zeta \in
\partial_\infty\pi_1(X)$, we have
$$\log[\alpha^+,\beta^+,\alpha^-,\zeta]-\log[\alpha^+,\beta^-,\alpha^-,\zeta]=G^r(P).$$
A similar equation holds for the first type gap.


\section{Applications to real and complex hyperbolic spaces}
In this section we prove
\begin{theorem}\label{main}Let $\Gamma\subset SL(2,\R)=SO^0(2,1)$ be a discrete group whose
quotient of $H^2_\R$ is a hyperbolic surface and its truncation of
flaring ends along closed geodesics is a Riemann surface $X$ with
$\alpha$ a boundary component. Let $G$ be a rank one semi-simple Lie
group associated with a division ring $R=\mathbb{R},\mathbb{C}$. If
$\rho_t:\Gamma\ra G$ is a type preserving deformation of $\Gamma$
for small $t$, then the following generalized McShane's identity
holds.
$$\ell(\alpha)=\sum_{P\in \cal P_\alpha}G(P)+\sum_{P\in \cal
S_\alpha}G^r(P)$$ where $\cal P_\alpha$ is the set of pants
$(\alpha,\beta,\gamma)$ so that
$(\alpha,\beta,\gamma)=(\alpha,\alpha^n\beta\alpha^{-n},\alpha^n\gamma\alpha^{-n})$
and $\cal S_\alpha$ is the set of  pants $(\alpha,\beta,\gamma)$ so
that either $\beta$ or $\gamma$ is a boundary of $X$, say $\beta$.
The gap functions are defined as
$$G(P)=\log[\alpha^+,\gamma^-,\alpha^-,\beta^+]$$
$$G^r(P)=\log[\alpha^+,\beta^+,\alpha^-,\beta^-].$$
\end{theorem}
\begin{proof}For $\rho_t$, identify $\partial_\infty\pi_1(X)$ with a
limit set of $\rho_t(\Gamma)$ via a homeomorphism. Note that when
$t=0$, the limit set lies on the round circle $\partial H^2_\R$.
Define $B:\partial_\infty\pi_1(X)\setminus \{\alpha^\pm\}\ra R$,
where $R=\R$ or $\bc$, by
$$B(y)=[\alpha^+,y,\alpha^-,\zeta]$$ where $\zeta\in \partial_\infty\pi_1(X)\setminus
\{\alpha^\pm\}$ is a fixed reference point.  Note that when
$y=\alpha^+$, $B(y)=0$ and $y=\zeta$, $B(y)=1$ and $y=\alpha^-$,
$B(y)=\infty$. So the image of $B$ lies on a curve in $R$ emanating
from the origin passing through $1$ and diverging to $\infty$, which
lies close to  $B(\partial_\infty \rho_0(\Gamma))$ when $t$ is
small. So take a line from the origin which does not intersect
$B(\partial_\infty \rho_0(\Gamma))$, to define $\log$. By
$$[\alpha^+,
\alpha(z),\alpha^-,\zeta]=[\alpha^+,z,\alpha^-,\zeta][\alpha^+,\alpha(z),\alpha^-,z],$$
$$\log B(\alpha(z))=\log B(z)-\ell(\alpha).$$
This means that the action of $\alpha$ on
$\partial_\infty\pi_1(X)\setminus \{\alpha^\pm\}$ and the image of
$B$ is related as above. So the fundamental domain $D$ of
$\partial_\infty\pi_1(X)\setminus \{\alpha^\pm\}$ under the action
of $\alpha$ can be identified, via $\log B$, to a curve connecting
$\log B(z)$ and $\log B(z)+\ell(\alpha)$ for any $z$. This map $\log
B$ is H\"older on $D$.

Also note that from
$$\frac{[\alpha^+,\gamma^-,\alpha^-,\zeta]}{[\alpha^+,\beta^+,\alpha^-,\zeta]}=[\alpha^+,\gamma^-,\alpha^-,\beta^+],$$
we get
$$\log B(\gamma^-)-\log B(\beta^+)=G(P).$$ Similar for $G^r$. This
means that if $(\beta^+,\gamma^-)$ is a gap, then their images under
$\log B$ differ by $G(P)$ for a pants $(\alpha,\beta,\gamma)$.

In conclusion since $\log B$ is H\"older, $\log B(K_\alpha^*)$ has
 zero length and so if we sum up all $\log B(\gamma^-)-\log
B(\beta^+)$ and $\log B(\beta^+)-\log B(\beta^-)$ for the gaps in
the fundamental domain $D$, we should obtain $\ell(\alpha)$, i.e.,
$$\ell(\alpha)=\sum_{P\in \cal P_\alpha}G(P)+\sum_{P\in \cal
S_\alpha}G^r(P).$$
\end{proof}

When the representations are fuchsian, we recover original
McShane-Mirzakhani identity.
 If $\alpha$ represents a cusp, then $G(P)=0$ since
 $\alpha^+=\alpha^-$. So we need an alternative approach.
This case is dealt with in section 4.2 of \cite{LM}. Define
$$W_\alpha(s,t)=\frac{\partial_y \log [\alpha^+,s,y,t]}{\partial_y \log[\alpha^+,s_0,y,\alpha(s_0)]}|_{y=\alpha^+}.$$
 Then $W_\alpha(s,t)$ is independent of $s_0$ and satisfies
\begin{enumerate}
\item[(a)]\label{(a)} $W_\alpha(\alpha(s),\alpha(t))=W_\alpha(s,t)$
\item[(b)] \label{(b)}$W_\alpha(s,\alpha(s))=1$
\item[(c)]\label{(c)} $W_\alpha(s,u)=W_\alpha(s,t)+W_\alpha(t,u).$
\end{enumerate}
 Then one defines
the cusp gap function for the pants $P=(\alpha,\beta,\gamma)$ with
the first boundary $\alpha$ a cusp,
$$W(P)=W_\alpha(\gamma^-,\beta^+),$$ and if $\beta$ is either
peripheral or a cusp then
$$W^r(P)=W_\alpha(\beta^+,\beta^-).$$
Then we get:
\begin{theorem}\label{main2}If $\alpha$ represents a cusp, then
$$1=\sum_{P\in \cal{P}_\alpha}W(P)+\sum_{P\in\cal S_\alpha}W^r(P).$$
\end{theorem}
\begin{proof}
One define a similar map $B:\partial_\infty \pi_1(X)\setminus
\{\alpha^+\} \ra \R (\text{or}\ \bc)$ by
$$B(y)=W_\alpha(s,y),$$ where $s\in \partial_\infty \pi_1(X)\setminus
\{\alpha^+\} $ is a fixed point. By the equations (b) and (c),
   $$B(\alpha(y))-B(y)=W_\alpha(s,\alpha(y))-W_\alpha(s,y)=W_\alpha(y,\alpha(y))=1.$$
In rank one case, it is easy to check that $B$ is still H\"older for
geometrically finite case, so the sum of gaps  between $y$ and
$\alpha(y)$ is equal to $1$, i.e.,
$$1=\sum_{P\in \cal{P}_\alpha}W(P)+\sum_{P\in\cal S_\alpha}W^r(P).$$
\end{proof}
This cuspidal case has many corollaries which cover known results so
far. We will discuss this in next section.

\section{Applications}
\subsection{$PSL(2,\bc)$ case}
The cross-ratio on $\partial H^3_\R=\hat \bc$ is in most simple
form
$$[a,b,c,d]=\frac{(a-b)(c-d)}{(a-d)(c-b)}$$ for $a,b,c,d\in \bc$.
This cross-ratio has another properties which does not share with
other cases. Following sections will make use of this extra
properties strongly. Any hyperbolic isometry $\gamma$ can be
conjugated into a form, whose action on $\hat\bc$ is
$$z\ra  re^{i\theta}z, \ r>1$$ so that $\log r+i\theta$ is called a complex
length of $\gamma$. This hyperbolic isometry has the repelling fixed
point $0$ and the attracting fixed point $\infty$. Then the period
$\ell(\gamma)$ of $\gamma$ is given by
$$\log (re^{i\theta})$$ which is a complex length of $\gamma$.

\subsubsection{Explicit calculation}
Let $P=(\alpha,\beta,\gamma)$ be a pair of pants so that
$\alpha\gamma\beta=1$. Then the shear coordinates for $P$ is
$$A=-[\gamma^-,\alpha^+,\beta^+,\gamma^{-1}(\alpha^+)]=-[\gamma^-,\alpha^+,\beta^+,\beta(\alpha^+)]$$
$$B=-[\alpha^+,\beta^+,\gamma^-,\alpha^{-1}(\beta^+)]$$
$$C=-[\beta^+,\gamma^-,\alpha^+,\beta^{-1}(\gamma^-)].$$
 A direct calculation shows that
 $W(P)=[\alpha(\gamma^-),\alpha^+,\beta^+,\gamma^-]$ using
 $s_0=\gamma^-$ in the definition of $W_\alpha$.
 By manipulating the cross-ratio we get
 $$W(P)=\frac{1}{1+e^{\frac{\ell(\beta)+\ell(\gamma)}{2}}}.$$
 Similarly
 $$W^r(P)=\frac{\sinh \frac{\ell(\beta)}{2}}{\cosh \frac{\ell(\gamma)}{2}+\cosh \frac{\ell(\beta)}{2}}.$$

\subsubsection{quasifuchsian case}
Let $X$ be a finite volume complete hyperbolic surface with $\alpha$
a cusp. Then original McShane's identity reads:
$$\sum_{\beta,\gamma}\frac{2}{1+e^{\frac{l(\beta)+l(\gamma)}{2}}}=1,$$
where $\beta\cup\gamma$ together with $\alpha$ bounds a pair of
pants. In our notation, $(\alpha,\beta,\gamma)$ and
$(\alpha,\gamma,\beta)$ are different pairs of pants so that the
factor 2 appears in the summand. Our result in Theorem \ref{main2}
shows that this equality holds for any type preserving quasifuchsian
representations where the real length $l(\beta)$ is replaced by its
complex length $\ell(\beta)$. This is obtained by \cite{AMS}(
Theorem 2.2) by a different method.

\subsubsection{Mapping torus case} Let $M$ be a complete hyperbolic
3-manifold which fiber over a circle with a puntured torus fibre
$T$. By Thurston this manifold is a mapping torus whose monodromy is
a pseudo-anasov  map $\phi$ on $T$. Then a limit set of a subgroup
$\Gamma_T$ corresponding to $T$ is whole Riemann sphere $\hat\bc$.
Let $\alpha$ denote the cusp in $\pi_1(T)$ so that
$\alpha^+=\alpha^-$. Take a sequence of quasifuchsian groups
$Q_n=(\phi^{-n}X, \phi^n X)$ which converges to $\Gamma_T$. For each
$n$, we have $\sum_{\beta}\frac{2}{1+e^{\ell(\beta)}}=1$. Since
$\beta$ and $\phi^i(\beta)$ represent  the same loop in $M$,
$\ell(\beta)=\ell(\phi^i(\beta))$ for all $i$ in $M$. This implies
that for all $Q_n$
$$\sum_{\beta\in\cal S/\phi}\sum_{i\in \mathbb Z} \frac{2}{1+e^{\ell(\phi^i\beta)}}=1,$$
hence we must have in $M$
$$\sum_{\cal S/\phi}\frac{2}{1+e^{\ell(\beta)}}=0,$$ where $\cal S$
denote the set of simple closed curves in $T$ so that $\cal S/\phi$
is the set of simple closed geodesics in $M$. This is the original
form of the result by \cite{bowditch1997t}. A similar result obtained by
\cite{AMS} also holds by our unified approach.

Let $M$ be a mapping torus with fiber $F$ a general punctured
surface glued by a pseudoanasov map $\phi$. If $\cal S$ denotes the
set of free homotopy class of simple closed curves on $F$, then
$\cal S/\phi$ denotes the set of free homotopy class of simple
closed curves on $M$ coming from the fiber $F$.

Then $Q_n\ra Q_\infty$ strongly where $Q_\infty$ corresponds to the
fiber group of $M$. Note that $\phi$ acts on $Q_\infty$ as an
isometry to give $M=Q_\infty/\phi$. Let $\ell_n(\gamma)$ be the
complex length of $\gamma\in \pi_1(F)$ in $Q_n$. Then we will show
that
\begin{proposition}For hyperbolic mapping torus $M$ with a fiber $F$ once punctured surface with a puncture $\alpha$,
$$\sum_{P=(\alpha,\beta,\gamma),\beta,\gamma\in \cal
S/\phi}\frac{1}{1+e^{\frac{\ell_M(\beta)+\ell_M(\gamma)}{2}}}=0.$$
\end{proposition}
\begin{proof}By geometric convergence, for any compact exhaustion $K_i\subset
Q_\infty$ of $Q_\infty$, there is $N_i$ such that there exists an
smooth embedding
$$f_n^i:K_i \ra Q_n,\ n>N_i$$ whose quasi-geodesic constant $1\leq
C_n^i <2$. Then for any $\gamma\in \cal S$ whose geodesic
representative is in $K_i$,
\begin{eqnarray}\label{comparison}
 \frac{1}{2} l_{Q_n}(\gamma)\leq l_M(\gamma)\leq 2
l_{Q_n}(\gamma),\ n>N_i.\end{eqnarray}  Let $\cal S_{K_i}$ be the
set of $\gamma\in\cal S$ whose geodesic representative lies in
$K_i$. By geometric convergence, such $\gamma$ has a geodesic
representative in a small neighborhood of $f_n^i(K_i)$ in $Q_n$ for
$n> N_i$.

Now fix $K_i$ for some $i$. For a fixed $n>N_i$, since $Q_n$ is
quasifuchsian there are $k>1$ and a fuchsian group $\rho_0$ so that
$$ \frac{1}{k}l_{\rho_0}(\gamma)\leq l_{Q_n}(\gamma)\leq k
l_{\rho_0}(\gamma), \forall \gamma.$$
 In conclusion by Equation (\ref{comparison}) there exist $k>1$ and a fuchsian group $\rho_0$ so
 that
$$\frac{1}{k}l_{\rho_0}(\gamma)\leq l_{M}(\gamma)\leq k
l_{\rho_0}(\gamma), \forall \gamma\in \cal S_{K_i}.$$ Then again by
Equation (\ref{comparison}), abusing notation with the same $k$,
\begin{eqnarray}\label{comparison2}
 \frac{1}{k}l_{\rho_0}(\gamma)\leq l_{Q_n}(\gamma)\leq k
l_{\rho_0}(\gamma), \forall \gamma, \forall n>N_i\ \text{and}\
n=\infty
\end{eqnarray}

In \cite{BS}, it is shown that for $F=H^2/\rho_0$, there exists $A$
consisting of finitely many mutually disjoint simple complete
geodesics each of which joins punctures of $F$ such that for
$\gamma\in\cal S$, if $||\gamma||$ is the intersection number of
$\gamma$ and $A$, then there are $c$ and a polynomial $P(n)$ such
that $l_{\rho_0}(\gamma)\geq c||\gamma||$ and $|S(n)=\{\gamma\in\cal
S|||\gamma||=n\}|\leq P(n)$. Then for any $n>N_i$ and $n=\infty$,
$$l_{Q_n}(\gamma)\geq \frac{1}{k}l_{\rho_0}(\gamma)\geq
\frac{c}{k}||\gamma||=\frac{c}{k}n, \forall \gamma\in S(n)\cap \cal
S_{K_i}.$$ For a pant $P$ whose boundaries are $\alpha,\beta,\gamma$
with $\alpha$ a fixed cusp,
$$|W(P)|=|\frac{1}{1+e^{\frac{\ell(\beta)+\ell(\gamma)}{2}}}|\leq \frac{1}{|e^{\frac{\ell(\beta)+\ell(\gamma)}{2}}|-1}$$
$$\leq\frac{1}{e^{\frac{l(\beta)+l(\gamma)}{2}}-1}\leq \frac{1}{e^{\frac{1}{2}l(\beta)}-1}\frac{1}{e^{\frac{1}{2}l(\gamma)}}.$$

By inequality (\ref{comparison2}), for any $n>N_i$ and $n=\infty$,
since there are only finitely many $S_0$ of $\gamma$ so that
$l_{\rho_0}(\gamma)< 2k\log 2$, except finitely many $\gamma$'s,
$l_{Q_n}(\gamma)\geq 2\log 2$. Hence outside $S_0$
$$\frac{1}{e^{\frac{1}{2}l(\beta)}-1}\leq 2
\frac{1}{e^{\frac{1}{2}l(\beta)}}.$$ Note that for any $n>N_i$ and
$n=\infty$,
$$\sum_{\gamma\in \cal S_{K_i}}
\frac{1}{e^{\frac{1}{2}l(\gamma)}}=\sum_{i=1}^\infty\sum_{\gamma\in
\cal S_{K_i}\cap S(n)} \frac{1}{e^{\frac{1}{2}l(\gamma)}}\leq
\sum_{i=1}^\infty \frac{P(n)}{e^{\frac{cn}{2k}}}=P_i<\infty.$$

Let $\cal P_\alpha^i$ be the set of pairs of pants
$P=(\alpha,\beta,\gamma)$ so that $\beta,\gamma\in S_{K_i}$.

Given $\epsilon>0$, choose a finite set $\cal S_{K_i}^\epsilon$
containing $S_0$ so that for any $n>N_i,n=\infty$, $$\sum_{\cal
S_{K_i}\setminus \cal S_{K_i}^\epsilon}
\frac{1}{e^{\frac{1}{2}l_{Q_n}(\gamma)}}<\epsilon.$$ Let $\cal
P_\alpha^\epsilon$ be a subset of $\cal P_\alpha^i$ consisting of
$P$'s so that $\beta,\gamma\in\cal S_{K_i}^\epsilon$. Then for any
$n>N_i$ and $n=\infty$,
$$\sum_{P\in\cal P_\alpha^i}|W(P)|\leq \sum_{P\in\cal P_\alpha^\epsilon}|W(P)| +\sum_{P\in\cal P_\alpha^i\setminus \cal P_\alpha^\epsilon} |W(P)|$$
$$\leq \sum_{P\in\cal P_\alpha^\epsilon}|W(P)| +\sum_{P\in\cal P_\alpha^i\setminus \cal P_\alpha^\epsilon}
2
\frac{1}{e^{\frac{1}{2}l_{Q_n}(\beta)}}\frac{1}{e^{\frac{1}{2}l_{Q_n}(\gamma)}}$$
$$\leq \sum_{P\in\cal P_\alpha^\epsilon}|W(P)| +
2(\sum_{\cal S_{K_i}\setminus \cal
S_{K_i}^\epsilon}\frac{1}{e^{\frac{1}{2}l_{Q_n}(\beta)}})(\sum_{\cal
S_{K_i}}\frac{1}{e^{\frac{1}{2}l_{Q_n}(\gamma)}})$$
$$\leq \sum_{P\in\cal P_\alpha^\epsilon}|W(P)| + 2\epsilon P_i.$$
Since $\cal P_\alpha^\epsilon$ is a finite set, this sum is
uniformly bounded for all $n>N_i$ and $n=\infty$. Hence for a given
$\epsilon$, we can choose a finite set $\cal
P^i_{\alpha,\epsilon}\subset \cal P^i_\alpha$ so that for any
$n>N_i,n=\infty$
$$|\sum_{\cal P^i_\alpha\setminus \cal
P^i_{\alpha,\epsilon}} W_n(P)|<\epsilon,$$ where $W_n$ is measured
in $Q_n$. Since $\cal P^i_{\alpha,\epsilon}$ is a finite set, we can
choose $n_i>N_i$ so that $$|\sum_{\cal P^i_{\alpha,\epsilon}}
W_\infty(P)-\sum_{\cal P^i_{\alpha,\epsilon}}
W_{n_i}(P)|<\epsilon,$$ hence
\begin{eqnarray}\label{com}
|\sum_{\cal P_{\alpha}^i}W_\infty(P) -\sum_{\cal
P_{\alpha}^i}W_{n_i}(P)|<3\epsilon.
\end{eqnarray}

Since $M=Q_\infty/\phi$, let $K=K_0$ be the fundamental domain of
$M$ in $Q_\infty$ and $K_i$ is the union of $K\cup \phi(K)\cup
\cdots \cup \phi^{i-1}(K)$. Note that $\cal S/\phi$ is identified
with the free homotopy classes of simple loops in $M$. Since $\phi$
is an isometry of $Q_\infty$,
$l_{Q_n}(\gamma)=l_{Q_n}(\phi^j\gamma)$ for any $j$, hence
$$\sum_{\cal P^i_\alpha} W_\infty(P)=i\sum_{\cal P^0_\alpha}
W_\infty(P).$$

For a given $\epsilon$, by Equation (\ref{com})
  $$|\sum_{\cal
P^i_\alpha} W_{n_i}(P)- i\sum_{\cal P^0_\alpha}
W_\infty(P)|<3\epsilon.$$  As $i\ra \infty$, $P^i_\alpha\ra
P_\alpha$ and the left term $|\sum_{\cal P^i_\alpha} W_{n_i}(P)|$ is
bounded above near $1$, hence
$$\sum_{P=(\alpha,\beta,\gamma),\beta,\gamma\in \cal
S/\phi}\frac{2}{1+e^{\frac{\ell_M(\beta)+\ell_M(\gamma)}{2}}}=0.$$
\end{proof}
\subsection{$PSU(2,1)$ case}
This is the case where no previous  results exist. Our chosen
complex valued cross-ratio is: for $z_1,z_2,z_3,z_4$ on the boundary
of unit ball model representing four points $a,b,c,d$, and $\tilde
z_1=(z_1,1),\tilde z_2=(z_2,1),\tilde z_3=(z_3,1),\tilde
z_4=(z_4,1)$  lifts in paraboloid model, then
$$[a,b,c,d]=\frac{\langle\tilde z_1,\tilde z_2\rangle \langle\tilde z_3,\tilde z_4\rangle}
{\langle\tilde z_1,\tilde z_4\rangle\langle\tilde z_3,\tilde
z_2\rangle}.$$

A general hyperbolic isometry in $H^2_\bc$ is in the form fixing $0$
and $\infty$:
$$[t,z]\ra [r^2t,re^{i\theta}z]$$ in Heisenberg coordinates.
A direct calculation shows that the period of such a hyperbolic
isometry $\gamma$ is
$$\ell(\gamma)=\log r^2$$ which is a real translation length of
$\gamma$. Somehow the period does not capture the rotational part.
We want to derive formulae for cuspidal and boundary loop case as in
fuchsian case.

\subsubsection{Gap functions}
Let $P=(\alpha,\beta,\gamma)$ be a pair of pants so that
$\alpha\gamma\beta=1$. We may normalize so that $\alpha$ fixes
$\infty$ and $0$, i.e. $\alpha^+=\infty$ and $\alpha^-=0$. Then,
as a matrix point of view,
$$
\alpha=E(\lambda)=\left[\begin{matrix} e^\lambda & 0 & 0 \\ 0 & e^{\overline{\lambda}-\lambda} & 0 \\
0 & 0 & e^{-\overline{\lambda}}
\end{matrix}\right],
$$
where $\lambda \in S$ and $e^\lambda$ is an attracting eigenvalue
of $\alpha$. Also for $Q,R \in SU(2,1)$ and  $\mu,\nu \in S$, we
can write

$$
\gamma=QE(\mu)Q^{-1}=\left[\begin{matrix} a & b & c \\ d & e & f \\
g & h & j
\end{matrix} \right] \left[\begin{matrix} e^\mu & 0 & 0 \\ 0 & e^{\overline{\mu}-\mu} & 0 \\
0 & 0 & e^{-\overline{\mu}} \end{matrix}\right] \left[\begin{matrix} \overline{j} & \overline{f} & \overline{c} \\
\overline{h} & \overline{e} & \overline{b} \\ \overline{g}
&\overline{d} & \overline{a}
\end{matrix}\right],
$$

$$
\beta=\gamma^{-1}\alpha^{-1}=RE(\nu)R^{-1}=\left[\begin{matrix} a' & b' & c' \\ d' & e' & f' \\
g' & h' & j'
\end{matrix} \right] \left[\begin{matrix} e^\nu & 0 & 0 \\ 0 & e^{\overline{\nu}-\nu} & 0 \\
0 & 0 & e^{-\overline{\nu}} \end{matrix}\right] \left[\begin{matrix} \overline{j'} & \overline{f'} & \overline{c'} \\
\overline{h'} & \overline{e'} & \overline{b'} \\ \overline{g'}
&\overline{d'} & \overline{a'}
\end{matrix}\right],
$$
where $e^\mu$ and $e^\nu$ are attracting eigenvalues of $\gamma$
and $\beta$ respectively. Furthermore, using some identities from $QQ^{-1}=I$, a direct calculation shows the following Lemma.

\begin{lemma}(Lemma 6.3 in \cite{PP})
$$
\hspace{-10ex}\beta=\left[\begin{matrix} e^{\overline{\nu}-\nu}+a'\overline{j'}\sigma(\nu)+c'\overline{g'}\sigma(-\overline{\nu}) & a'\overline{f'}\sigma(\nu)+c'\overline{d'}\sigma(-\overline{\nu}) & a'\overline{c'}\sigma(\nu)+c'\overline{a'}\sigma(-\overline{\nu}) \\ d'\overline{j'}\sigma(\nu)+f'\overline{g'}\sigma(-\overline{\nu}) & e^{\overline{\nu}-\nu}+d'\overline{f'}\sigma(\nu)+f'\overline{d'}\sigma(-\overline{\nu}) & d'\overline{c'}\sigma(\nu)+f'\overline{a'}\sigma(-\overline{\nu}) \\ g'\overline{j'}\sigma(\nu)+j'\overline{g'}\sigma(-\overline{\nu})
& g'\overline{f'}\sigma(\nu)+j'\overline{d'}\sigma(-\overline{\nu}) & e^{\overline{\nu}-\nu}+g'\overline{c'}\sigma(\nu)+j'\overline{a'}\sigma(-\overline{\nu})
\end{matrix}\right],
$$
where $\sigma(\nu)=e^\nu-e^{ \overline{\nu}-\nu}$.
\end{lemma}
Also we can write $\beta$ as the following, using the identity $\alpha\gamma\beta=1$.

\begin{lemma}
Let $\alpha$, $\gamma$ and $\beta$ be defined as above. Then
$$
\hspace{-15ex}\beta=\left[\begin{matrix} e^{-\lambda}(e^{\mu-\overline{\mu}}+a\overline{j}\sigma(-\mu)+c\overline{g}\sigma(\overline{\mu})) & e^{\lambda-\overline{\lambda}}(a\overline{f}\sigma(-\mu)+c\overline{d}\sigma(\overline{\mu})) & e^{\overline{\lambda}}(a\overline{c}\sigma(-\mu)+c\overline{a}\sigma(\overline{\mu})) \\ e^{-\lambda}(d\overline{j}\sigma(-\mu)+f\overline{g}\sigma(\overline{\mu})) & e^{\lambda-\overline{\lambda}}(e^{\mu-\overline{\mu}}+d\overline{f}\sigma(-\mu)+f\overline{d}\sigma(\overline{\mu})) & e^{\overline{\lambda}}(d\overline{c}\sigma(-\mu)+f\overline{a}\sigma(\overline{\mu})) \\
e^{-\lambda}(g\overline{j}\sigma(-\mu)+j\overline{g}\sigma(\overline{\mu})) & e^{\lambda-\overline{\lambda}}(g\overline{f}\sigma(-\mu)+j\overline{d}\sigma(\overline{\mu})) & e^{\overline{\lambda}}(e^{\mu-\overline{\mu}}+g\overline{c}\sigma(-\mu)+j\overline{a}\sigma(\overline{\mu}))
\end{matrix} \right].
$$

\end{lemma}

\begin{proof}
The proof is direct from the following calculations.
\begin{align*}
\beta &= \gamma^{-1}\alpha^{-1}=QE(-\mu)Q^{-1}E(-\lambda)\\ &= \left[\begin{matrix} a & b & c \\ d & e & f \\
g & h & j
\end{matrix} \right] \left[\begin{matrix} e^{-\mu} & 0 & 0 \\ 0 & e^{\mu-\overline{\mu}} & 0 \\
0 & 0 & e^{\overline{\mu}} \end{matrix}\right] \left[\begin{matrix} \overline{j} & \overline{f} & \overline{c} \\
\overline{h} & \overline{e} & \overline{b} \\ \overline{g}
&\overline{d} & \overline{a}
\end{matrix}\right]\left[\begin{matrix} e^{-\lambda} & 0 & 0 \\ 0 & e^{\lambda-\overline{\lambda}} & 0 \\
0 & 0 & e^{\overline{\lambda}} \end{matrix}\right]
\end{align*}
$$
\hspace{-15ex}=\left[\begin{matrix} e^{-\lambda}(e^{-\mu}a\overline{j}+e^{\mu-\overline{\mu}}b\overline{h}+e^{\overline{\mu}}c\overline{g}) & e^{\lambda-\overline{\lambda}}(e^{-\mu}a\overline{f}+e^{\mu-\overline{\mu}}b\overline{e}+e^{\overline{\mu}}c\overline{d}) & e^{\overline{\lambda}}(e^{-\mu}a\overline{c}+e^{\mu-\overline{\mu}}b\overline{b}+e^{\overline{\mu}}c\overline{a}) \\ e^{-\lambda}(e^{-\mu}d\overline{j}+e^{\mu-\overline{\mu}}e\overline{h}+e^{\overline{\mu}}f\overline{g}) & e^{\lambda-\overline{\lambda}}(e^{-\mu}d\overline{f}+e^{\mu-\overline{\mu}}e\overline{e}+e^{\overline{\mu}}f\overline{d}) & e^{\overline{\lambda}}(e^{-\mu}d\overline{c}+e^{\mu-\overline{\mu}}e\overline{b}+e^{\overline{\mu}}f\overline{a}) \\
e^{-\lambda}(e^{-\mu}g\overline{j}+e^{\mu-\overline{\mu}}h\overline{h}+e^{\overline{\mu}}j\overline{g}) & e^{\lambda-\overline{\lambda}}(e^{-\mu}g\overline{f}+e^{\mu-\overline{\mu}}h\overline{e}+e^{\overline{\mu}}j\overline{d}) & e^{\overline{\lambda}}(e^{-\mu}g\overline{c}+e^{\mu-\overline{\mu}}h\overline{b}+e^{\overline{\mu}}j\overline{a})
\end{matrix} \right]
$$
$$
\hspace{-15ex}=\left[\begin{matrix} e^{-\lambda}(e^{\mu-\overline{\mu}}+a\overline{j}\sigma(-\mu)+c\overline{g}\sigma(\overline{\mu})) & e^{\lambda-\overline{\lambda}}(a\overline{f}\sigma(-\mu)+c\overline{d}\sigma(\overline{\mu})) & e^{\overline{\lambda}}(a\overline{c}\sigma(-\mu)+c\overline{a}\sigma(\overline{\mu})) \\ e^{-\lambda}(d\overline{j}\sigma(-\mu)+f\overline{g}\sigma(\overline{\mu})) & e^{\lambda-\overline{\lambda}}(e^{\mu-\overline{\mu}}+d\overline{f}\sigma(-\mu)+f\overline{d}\sigma(\overline{\mu})) & e^{\overline{\lambda}}(d\overline{c}\sigma(-\mu)+f\overline{a}\sigma(\overline{\mu})) \\
e^{-\lambda}(g\overline{j}\sigma(-\mu)+j\overline{g}\sigma(\overline{\mu})) & e^{\lambda-\overline{\lambda}}(g\overline{f}\sigma(-\mu)+j\overline{d}\sigma(\overline{\mu})) & e^{\overline{\lambda}}(e^{\mu-\overline{\mu}}+g\overline{c}\sigma(-\mu)+j\overline{a}\sigma(\overline{\mu}))
\end{matrix} \right].
$$
For the last equality, we use six identities which are from $QQ^{-1}=I$, for example $1=a\overline{j}+b\overline{h}+c\overline{g}$ for the top left-hand entry.
\end{proof}

By comparing above two lemmas, we get the following proposition, which will be used later.

\begin{proposition}
$$
\frac{\overline{a'}}{\overline{g'}}=\frac{e^{-\overline{\lambda}}(e^{\overline{\mu}-\mu}+\overline{a}j\sigma(-\overline{\mu})+\overline{c}g\sigma(\mu))+
e^{-\overline{\nu}}e^{\overline{\lambda}}(e^{\mu-\overline{\mu}}+g\overline{c}\sigma(-\mu)+j\overline{a}\sigma(\overline{\mu}))-e^{\nu-\overline{\nu}}-e^{-\nu}}
{e^{-\overline{\lambda}}(g\overline{j}\sigma(\mu)+j\overline{g}\sigma(-\overline{\mu}))+e^{-\overline{\nu}}e^{-\lambda}(g\overline{j}\sigma(-\mu)+j\overline{g}\sigma(\overline{\mu}))}.
$$
\end{proposition}

\begin{proof}
The left-bottom entry of $\beta$ is
$$
g'\overline{j'}\sigma(\nu)+j'\overline{g'}\sigma(-\overline{\nu})=e^{-\lambda}(g\overline{j}\sigma(-\mu)+j\overline{g}\sigma(\overline{\mu}))
$$
By conjugating, we get
$$
g'\overline{j'}\sigma(-\nu)+j'\overline{g'}\sigma(\overline{\nu})=e^{-\overline{\lambda}}(g\overline{j}\sigma(\mu)+j\overline{g}\sigma(-\overline{\mu}))
$$
We can rewrite these two equations as follows, using $\sigma(-\nu)=-e^{-\overline{\nu}}\sigma(\nu)$.
$$
g'\overline{j'}e^{-\overline{\nu}}\sigma(\nu)+j'\overline{g'}e^{-\overline{\nu}}\sigma(-\overline{\nu})
=e^{-\overline{\nu}}e^{-\lambda}(g\overline{j}\sigma(-\mu)+j\overline{g}\sigma(\overline{\mu})),
$$
$$
-g'\overline{j'}e^{-\overline{\nu}}\sigma(\nu)+j'\overline{g'}\sigma(\overline{\nu})
=e^{-\overline{\lambda}}(g\overline{j}\sigma(\mu)+j\overline{g}\sigma(-\overline{\mu})).
$$
By adding them, we have
$$
(\sigma(\overline{\nu})+e^{-\overline{\nu}}\sigma(-\overline{\nu}))j'\overline{g'}
=e^{-\overline{\lambda}}(g\overline{j}\sigma(\mu)+j\overline{g}\sigma(-\overline{\mu}))
+e^{-\overline{\nu}}e^{-\lambda}(g\overline{j}\sigma(-\mu)+j\overline{g}\sigma(\overline{\mu})).
$$
Hence,
$$
j'=\frac{1}{\overline{g'}}\frac{e^{-\overline{\lambda}}(g\overline{j}\sigma(\mu)+j\overline{g}\sigma(-\overline{\mu}))
+e^{-\overline{\nu}}e^{-\lambda}(g\overline{j}\sigma(-\mu)+j\overline{g}\sigma(\overline{\mu}))}{(1-e^{-(\nu+\overline{\nu})})\sigma(\overline{\nu})}.
$$
Similarly, if we compare the right-top entry of $\beta$, we get
$$
c'=\frac{1}{\overline{a'}}\frac{e^{\lambda}(a\overline{c}\sigma(\mu)+c\overline{a}\sigma(-\overline{\mu}))
+e^{-\overline{\nu}}e^{\overline{\lambda}}(a\overline{c}\sigma(-\mu)+c\overline{a}\sigma(\overline{\mu}))}{(1-e^{-(\nu+\overline{\nu})})\sigma(\overline{\nu})}.
$$
If we substitute these $j'$ and $c'$ to the left-top entry and right-bottom entry in Lemma 5.2, then we have
\begin{align*}
& e^{\overline{\nu}-\nu}+\frac{a'}{g'}\frac{e^{-\lambda}(\overline{g}j\sigma(\overline{\mu})+\overline{j}g\sigma(-\mu))
+e^{-\nu}e^{-\overline{\lambda}}(\overline{g}j\sigma(-\overline{\mu})+\overline{j}g\sigma(\mu))}{1-e^{-(\nu+\overline{\nu})}}\\
& -e^{-\nu}\frac{\overline{g'}}{\overline{a'}}\frac{e^{\lambda}(a\overline{c}\sigma(\mu)+c\overline{a}\sigma(-\overline{\mu}))
+e^{-\overline{\nu}}e^{\overline{\lambda}}(a\overline{c}\sigma(-\mu)+c\overline{a}\sigma(\overline{\mu}))}{1-e^{-(\nu+\overline{\nu})}}\\
& =e^{-\lambda}(e^{\mu-\overline{\mu}}+a\overline{j}\sigma(-\mu)+c\overline{g}\sigma(\overline{\mu}))
\end{align*}
and
\begin{align*}
& e^{\overline{\nu}-\nu}+\frac{g'}{a'}\frac{e^{\overline{\lambda}}(\overline{a}c\sigma(\overline{\mu})+\overline{c}a\sigma(-\mu))
+e^{-\nu}e^{\lambda}(\overline{a}c\sigma(-\overline{\mu})+\overline{c}a\sigma(\mu))}{1-e^{-(\nu+\overline{\nu})}}\\
& -e^{-\nu}\frac{\overline{a'}}{\overline{g'}}\frac{e^{-\overline{\lambda}}(g\overline{j}\sigma(\mu)+j\overline{g}\sigma(-\overline{\mu}))
+e^{-\overline{\nu}}e^{-\lambda}(g\overline{j}\sigma(-\mu)+j\overline{g}\sigma(\overline{\mu}))}{1-e^{-(\nu+\overline{\nu})}}\\
& =e^{\overline{\lambda}}(e^{\mu-\overline{\mu}}+g\overline{c}\sigma(-\mu)+j\overline{a}\sigma(\overline{\mu})).
\end{align*}

Now we conjugate the first one and multiply $e^{-\overline{\nu}}$ to the second one, then they are as follows respectively.
\begin{align*}
& e^{\nu-\overline{\nu}}-e^{-\overline{\nu}}\frac{g'}{a'}\frac{e^{\overline{\lambda}}(\overline{a}c\sigma(\overline{\mu})+\overline{c}a\sigma(-\mu))
+e^{-\nu}e^{\lambda}(\overline{a}c\sigma(-\overline{\mu})+\overline{c}a\sigma(\mu))}{1-e^{-(\nu+\overline{\nu})}}\\
& +\frac{\overline{a'}}{\overline{g'}}\frac{e^{-\overline{\lambda}}(g\overline{j}\sigma(\mu)+j\overline{g}\sigma(-\overline{\mu}))
+e^{-\overline{\nu}}e^{-\lambda}(g\overline{j}\sigma(-\mu)+j\overline{g}\sigma(\overline{\mu}))}{1-e^{-(\nu+\overline{\nu})}}\\
& =e^{-\overline{\lambda}}(e^{\overline{\mu}-\mu}+\overline{a}j\sigma(-\overline{\mu})+\overline{c}g\sigma(\mu))
\end{align*}
and
\begin{align*}
& e^{-\nu}+e^{-\overline{\nu}}\frac{g'}{a'}\frac{e^{\overline{\lambda}}(\overline{a}c\sigma(\overline{\mu})+\overline{c}a\sigma(-\mu))
+e^{-\nu}e^{\lambda}(\overline{a}c\sigma(-\overline{\mu})+\overline{c}a\sigma(\mu))}{1-e^{-(\nu+\overline{\nu})}}\\
&-e^{-\nu-\overline{\nu}}\frac{\overline{a'}}{\overline{g'}}\frac{e^{-\overline{\lambda}}(g\overline{j}\sigma(\mu)+j\overline{g}\sigma(-\overline{\mu}))
+e^{-\overline{\nu}}e^{-\lambda}(g\overline{j}\sigma(-\mu)+j\overline{g}\sigma(\overline{\mu}))}{1-e^{-(\nu+\overline{\nu})}}\\
& =e^{-\overline{\nu}}e^{\overline{\lambda}}(e^{\mu-\overline{\mu}}+g\overline{c}\sigma(-\mu)+j\overline{a}\sigma(\overline{\mu})).
\end{align*}
Finally, by adding above two equations, we get the result.
\end{proof}

If we define the following three cross-ratios\
$$
\X_1:=\X_1(\alpha,\gamma)=[r_\gamma, a_\alpha, r_\alpha,
a_\gamma],
$$
$$
\X_2:=\X_2(\alpha,\gamma)=[r_\gamma, r_\alpha, a_\alpha,
a_\gamma],
$$
$$\X_3:=\X_3(\alpha,\gamma)=[r_\gamma, r_\alpha, a_\gamma,
a_\alpha],
$$
where $a_\alpha$ and $r_\alpha$ are attracting and repelling fixed
points of $\alpha$ respectively and they are the same for
$\gamma$, by easy calculation(See Lemma 6.2 in \cite{PP}), we get
$\X_1=j\overline{a}$, $\X_2=c\overline{g}$, $\X_3=\frac{cg}{aj}$.

Now we are ready to calculate $G(P)$ and $G^r(P)$.\\
\begin{align*}
\hspace{-10ex}G(P) &= \log [\alpha^+, \gamma^-, \alpha^-, \beta^+]\\
&= \log [\infty, Q(0), 0, R(\infty)]\\
&= \log \frac{<\infty,Q(0)><0, R(\infty)>}{<\infty,R(\infty)><0,Q(0)>}\\
&= \log \frac{\left< \left(\begin{matrix} 1 \\ 0
\\ 0 \end{matrix}\right), \left(\begin{matrix} c
\\ f \\ j \end{matrix}\right) \right> \left< \left(\begin{matrix}
0 \\ 0 \\ 1\end{matrix}\right),  \left(\begin{matrix} a'
\\ d' \\ g'\end{matrix}\right) \right>}{\left< \left(\begin{matrix} 1
\\ 0 \\ 0\end{matrix}\right),  \left(\begin{matrix} a'
\\ d' \\ g'\end{matrix}\right)\right> \left< \left(\begin{matrix}
0 \\ 0 \\ 1\end{matrix}\right),  \left(\begin{matrix} c
\\ f \\ j\end{matrix}\right) \right>}\\
&= \log \frac{\overline{ja'}}{\overline{g'c}}\\
&= \log \frac{\overline{j}}{\overline{c}}\frac{e^{-\overline{\lambda}}(e^{\overline{\mu}-\mu}+\overline{a}j\sigma(-\overline{\mu})+\overline{c}g\sigma(\mu))+
e^{-\overline{\nu}}e^{\overline{\lambda}}(e^{\mu-\overline{\mu}}+g\overline{c}\sigma(-\mu)+j\overline{a}\sigma(\overline{\mu}))
-e^{\nu-\overline{\nu}}-e^{-\nu}}
{e^{-\overline{\lambda}}(g\overline{j}\sigma(\mu)+j\overline{g}\sigma(-\overline{\mu}))
+e^{-\overline{\nu}}e^{-\lambda}(g\overline{j}\sigma(-\mu)+j\overline{g}\sigma(\overline{\mu}))}\\
&= \log \frac{e^{-\overline{\lambda}}(e^{\overline{\mu}-\mu}+\X_1\sigma(-\overline{\mu})+\overline{\X_2}\sigma(\mu))+
e^{-\overline{\nu}}e^{\overline{\lambda}}(e^{\mu-\overline{\mu}}+\overline{\X_2}\sigma(-\mu)+\X_1\sigma(\overline{\mu}))-e^{\nu-\overline{\nu}}-e^{-\nu}}
{e^{-\overline{\lambda}}(\overline{\X_2}\sigma(\mu)+\X_1\overline{\X_3}\sigma(-\overline{\mu}))
+e^{-\overline{\nu}}e^{-\lambda}(\overline{\X_2}\sigma(-\mu)+\X_1\overline{\X_3}\sigma(\overline{\mu}))}.
\end{align*}

For the last line, we use that $\frac{\overline{c}j\overline{g}}{\overline{j}}=\X_1\overline{\X_3}$. Furthermore,
\begin{align*}
\hspace{-20ex}G^r(P) &= \log [\alpha^+, \beta^+, \alpha^-, \beta^-]\\
&= \log [\infty, R(\infty), 0, R(0)]\\
&= \log \frac{<\infty,R(\infty)><0, R(0)>}{<\infty,R(0)><0,R(\infty)>}\\
&= \log \frac{\left< \left(\begin{matrix} 1 \\ 0
\\ 0 \end{matrix}\right), \left(\begin{matrix} a'
\\ d' \\ g' \end{matrix}\right) \right> \left< \left(\begin{matrix}
0 \\ 0 \\ 1\end{matrix}\right),  \left(\begin{matrix} c'
\\ f' \\ j'\end{matrix}\right) \right>}{\left< \left(\begin{matrix} 1
\\ 0 \\ 0\end{matrix}\right),  \left(\begin{matrix} c'
\\ f' \\ j'\end{matrix}\right)\right> \left< \left(\begin{matrix}
0 \\ 0 \\ 1\end{matrix}\right),  \left(\begin{matrix} a'
\\ d' \\ g'\end{matrix}\right) \right>}\\
&= \log \frac{\overline{g'}\overline{c'}}{\overline{j'}\overline{a'}}\\
&= \log \frac{\overline{g'}}{\overline{a'}}\frac{g'(e^{\overline{\lambda}}(\overline{a}c\sigma(\overline{\mu})+\overline{c}a\sigma(-\mu))
+e^{-\nu}e^{\lambda}(\overline{a}c\sigma(-\overline{\mu})+\overline{c}a\sigma(\mu)))}
{a'(e^{-\lambda}(\overline{g}j\sigma(\overline{\mu})+\overline{j}g\sigma(-\mu))
+e^{-\nu}e^{-\overline{\lambda}}(\overline{g}j\sigma(-\overline{\mu})+\overline{j}g\sigma(\mu)))}\\
&= \log \frac{|g'|^2(e^{\overline{\lambda}}(\overline{a}c\sigma(\overline{\mu})+\overline{c}a\sigma(-\mu))
+e^{-\nu}e^{\lambda}(\overline{a}c\sigma(-\overline{\mu})+\overline{c}a\sigma(\mu)))}
{|a'|^2(e^{-\lambda}(\overline{g}j\sigma(\overline{\mu})+\overline{j}g\sigma(-\mu))
+e^{-\nu}e^{-\overline{\lambda}}(\overline{g}j\sigma(-\overline{\mu})+\overline{j}g\sigma(\mu)))}\\
&= \log [\frac{e^{\overline{\lambda}}(\overline{a}c\sigma(\overline{\mu})+\overline{c}a\sigma(-\mu))
+e^{-\nu}e^{\lambda}(\overline{a}c\sigma(-\overline{\mu})+\overline{c}a\sigma(\mu))}
{e^{-\lambda}(\overline{g}j\sigma(\overline{\mu})+\overline{j}g\sigma(-\mu))
+e^{-\nu}e^{-\overline{\lambda}}(\overline{g}j\sigma(-\overline{\mu})+\overline{j}g\sigma(\mu))}\\
& \cdot \frac{|e^{-\overline{\lambda}}(g\overline{j}\sigma(\mu)+j\overline{g}\sigma(-\overline{\mu}))
+e^{-\overline{\nu}}e^{-\lambda}(g\overline{j}\sigma(-\mu)+j\overline{g}\sigma(\overline{\mu}))|^2}
{|e^{-\overline{\lambda}}(e^{\overline{\mu}-\mu}+\overline{a}j\sigma(-\overline{\mu})+\overline{c}g\sigma(\mu))
+e^{-\overline{\nu}}e^{\overline{\lambda}}(e^{\mu-\overline{\mu}}+g\overline{c}\sigma(-\mu)+j\overline{a}\sigma(\overline{\mu}))
-e^{\nu-\overline{\nu}}-e^{-\nu}|^2}] \\
&= \log [e^{\overline{\lambda}}(\overline{a}c\sigma(\overline{\mu})+\overline{c}a\sigma(-\mu))
+e^{-\nu}e^{\lambda}(\overline{a}c\sigma(-\overline{\mu})+\overline{c}a\sigma(\mu))\\
& \cdot \frac{e^{-\overline{\lambda}}(g\overline{j}\sigma(\mu)+j\overline{g}\sigma(-\overline{\mu}))
+e^{-\overline{\nu}}e^{-\lambda}(g\overline{j}\sigma(-\mu)+j\overline{g}\sigma(\overline{\mu}))}
{|e^{-\overline{\lambda}}(e^{\overline{\mu}-\mu}+\overline{a}j\sigma(-\overline{\mu})+\overline{c}g\sigma(\mu))
+e^{-\overline{\nu}}e^{\overline{\lambda}}(e^{\mu-\overline{\mu}}+g\overline{c}\sigma(-\mu)+j\overline{a}\sigma(\overline{\mu}))
-e^{\nu-\overline{\nu}}-e^{-\nu}|^2}].
\end{align*}
Here, the denominator is
$$
|e^{-\overline{\lambda}}(e^{\overline{\mu}-\mu}+\X_1\sigma(-\overline{\mu})+\overline{\X_2}\sigma(\mu))+
e^{-\overline{\nu}}e^{\overline{\lambda}}(e^{\mu-\overline{\mu}}
+\overline{\X_2}\sigma(-\mu)+\X_1\sigma(\overline{\mu}))-e^{\nu-\overline{\nu}}-e^{-\nu}|^2
$$
because $\X_1=j\overline{a}$ and $\X_2=c\overline{g}$, and the numerator is
\begin{align*}
&((e^{\overline{\lambda}}\sigma(\overline{\mu})+e^{-\nu}e^{\lambda}\sigma(-\overline{\mu}))\overline{a}c
+(e^{\overline{\lambda}}\sigma(-\mu)+e^{-\nu}e^{\lambda}\sigma(\mu))\overline{c}a)\\
& \cdot
((e^{-\overline{\lambda}}\sigma(\mu)+e^{-\overline{\nu}}e^{-\lambda}\sigma(-\mu))g\overline{j}
+(e^{-\overline{\lambda}}\sigma(-\overline{\mu})+e^{-\overline{\nu}}e^{-\lambda}\sigma(\overline{\mu}))j\overline{g})\\
&= |\sigma(\mu)(1-e^{-\nu}e^{-\mu}e^{\lambda-\overline{\lambda}})|^2\overline{a}cg\overline{j}\\
&+\sigma(\overline{\mu})^2(1-e^{-\nu}e^{-\mu}e^{\lambda-\overline{\lambda}})
(e^{-\overline{\nu}}e^{\overline{\lambda}-\lambda}-e^{-\mu})\overline{a}cj\overline{g}\\
&+\sigma(\mu)^2\overline{(1-e^{-\nu}e^{-\mu}e^{\lambda-\overline{\lambda}})}
\overline{(e^{-\overline{\nu}}e^{\overline{\lambda}-\lambda}-e^{-\mu})}a\overline{c}g\overline{j}\\
&+|\sigma(\mu)(e^{-\overline{\nu}}e^{\overline{\lambda}-\lambda}-e^{-\mu})|^2a\overline{c}j\overline{g}\\
&= |\sigma(\mu)(1-e^{-\nu}e^{-\mu}e^{\lambda-\overline{\lambda}})|^2|\X_1|^2\X_3
+|\sigma(\mu)(e^{-\overline{\nu}}e^{\overline{\lambda}-\lambda}-e^{-\mu})|^2|\X_1|^2\overline{\X_3}\\
&+2Re[\sigma(\overline{\mu})^2(1-e^{-\nu}e^{-\mu}e^{\lambda-\overline{\lambda}})
(e^{-\overline{\nu}}e^{\overline{\lambda}-\lambda}-e^{-\mu})\X_1\X_2]
\end{align*}
because $\overline{a}cg\overline{j}=|\X_1|^2\X_3$.
We can also express $G(P)$ and $G^r(P)$ as a function of $\X_1, \X_2, \lambda,
\mu, \nu$ and $tr[A,B]$ because $\X_3$ can be written as a function of
$\X_1, \X_2, \lambda, \mu, \nu$ and $tr[A,B]$ by Corollary 6.5 in
\cite{PP}. Furthermore, by Proposition 7.6 in \cite{PP}, it is also
possible to write $G(P)$ and $G^r(P)$ as a function of $\lambda, \mu, \nu, tr[A,B],
tr(AB)$ and $tr(A^{-1}B)$.

\subsubsection{Cusp Gap functions}

In this section, we calculate cusp gap functions in fuchsian case, which has already done in \cite{LM} but here we use complex hyperbolic coordinates and give a new proof.\\
When $\alpha$ represents a cusp, we normalize so that the fixed
point of $\alpha$, say $\alpha^{+}$, is $\infty$. In fuchsian
case, all fixed points of $\alpha$, $\beta$, and $\gamma$ are on
the $t$-axis in Heisenberg group. Then as a matrix point of view,
$$
\alpha=\left[\begin{matrix} 1 & 0 & it \\ 0 & 1 & 0 \\
0 & 0 & 1
\end{matrix}\right],
$$
where $t$ is a real number and $i=\sqrt{-1}$. As in the above section, for $Q,R \in SU(2,1)$ and  $\mu,\nu \in S$, we
can write

$$
\gamma=QE(\mu)Q^{-1}=\left[\begin{matrix} a & b & c \\ d & e & f \\
g & h & j
\end{matrix} \right] \left[\begin{matrix} e^\mu & 0 & 0 \\ 0 & e^{\overline{\mu}-\mu} & 0 \\
0 & 0 & e^{-\overline{\mu}} \end{matrix}\right] \left[\begin{matrix} \overline{j} & \overline{f} & \overline{c} \\
\overline{h} & \overline{e} & \overline{b} \\ \overline{g}
&\overline{d} & \overline{a}
\end{matrix}\right],
$$

$$
\beta=\gamma^{-1}\alpha^{-1}=RE(\nu)R^{-1}=\left[\begin{matrix} a' & b' & c' \\ d' & e' & f' \\
g' & h' & j'
\end{matrix} \right] \left[\begin{matrix} e^\nu & 0 & 0 \\ 0 & e^{\overline{\nu}-\nu} & 0 \\
0 & 0 & e^{-\overline{\nu}} \end{matrix}\right] \left[\begin{matrix} \overline{j'} & \overline{f'} & \overline{c'} \\
\overline{h'} & \overline{e'} & \overline{b'} \\ \overline{g'}
&\overline{d'} & \overline{a'}
\end{matrix}\right],
$$
where $e^\mu$ and $e^\nu$ are attracting eigenvalues of $\gamma$
and $\beta$ respectively. Then, by a direct calculation, the periods of $\beta$ and $\gamma$ are
$$
\ell(\beta)=\log[\beta^-,\beta(y),\beta^+,y]=\log e^{\lambda+\overline{\lambda}}=\lambda+\overline{\lambda} \quad
and \quad \ell(\gamma)=\mu+\overline{\mu} \quad (mod \quad 2\pi i).
$$

Since the fixed points of $\beta$ and $\gamma$ are on the $t$-axis in Heisenberg group,
$$
\gamma^+=\left(\begin{matrix} a \\ d \\ g \end{matrix}\right) \sim \left(\begin{matrix} it_1 \\ 0 \\ 1 \end{matrix}\right),
\quad \gamma^-=\left(\begin{matrix} c \\ f \\ j \end{matrix}\right) \sim \left(\begin{matrix} it_2 \\ 0 \\ 1 \end{matrix}\right),
$$
where $t_1$ and $t_2$ are distinct real numbers. By using some
identities from $QQ^{-1}=I$, we can show that $Q$ must be of the
form
$$
Q=\left[\begin{matrix} it_1g & 0 & it_2j \\ 0 & -i\overline{gj}(t_1-t_2) & 0 \\
g & 0 & j
\end{matrix} \right],
$$
where $g\overline{j}(t_1-t_2)=-i$, $j\overline{g}(t_1-t_2)=i$, and $\mid gj(t_1-t_2) \mid =1$.\\
Similarly,
$$
R=\left[\begin{matrix} is_1g' & 0 & is_2j' \\ 0 & -i\overline{g'j'}(s_1-s_2) & 0 \\
g' & 0 & j'
\end{matrix} \right],
$$
where $s_1$ and $s_2$ are distinct real numbers(they cannot be the same as $t_1$ and $t_2$ as well.) and $g'\overline{j'}(s_1-s_2)=-i$, $j'\overline{g'}(s_1-s_2)=i$, and $\mid g'j'(s_1-s_2) \mid =1$.\\
Then,
\begin{align*}
\beta &= \gamma^{-1}\alpha^{-1}=QE(-\mu)Q^{-1}\alpha^{-1}\\ &= \left[\begin{matrix} it_1g & 0 & it_2j \\ 0 & -i\overline{gj}(t_1-t_2) & 0 \\
g & 0 & j
\end{matrix} \right] \left[\begin{matrix} e^{-\mu} & 0 & 0 \\ 0 & e^{\mu-\overline{\mu}} & 0 \\
0 & 0 & e^{\overline{\mu}} \end{matrix}\right] \left[\begin{matrix} \overline{j} & 0 & -it_2\overline{j} \\
0 & igj(t_1-t_2) & 0 \\ \overline{g}
& 0 & -it_1\overline{g}
\end{matrix}\right]\left[\begin{matrix} 1 & 0 & -it \\ 0 & 1 & 0 \\
0 & 0 & 1 \end{matrix}\right]\\
&= \left[\begin{matrix} it_1ge^{-\mu} & 0 & it_2je^{\overline{\mu}} \\ 0 & -i\overline{gj}(t_1-t_2)e^{\mu-\overline{\mu}} & 0 \\
e^{-\mu}g & 0 & e^{\overline{\mu}}j
\end{matrix} \right] \left[\begin{matrix} \overline{j} & 0 & -i\overline{j}(t+t_2) \\ 0 & igj(t_1-t_2) & 0 \\
\overline{g} & 0 & -i\overline{g}(t+t_1) \end{matrix}\right]\\
&= \left[\begin{matrix} it_1e^{-\mu}g\overline{j}+it_2e^{\overline{\mu}}j\overline{g} & 0 & t_1(t+t_2)e^{-\mu}g\overline{j}+t_2(t+t_1)e^{\overline{\mu}}j\overline{g} \\ 0 &  e^{\mu-\overline{\mu}} & 0 \\
e^{-\mu}g\overline{j}+e^{\overline{\mu}}j\overline{g} & 0 & -i(t+t_2)e^{-\mu}g\overline{j}-i(t+t_1)e^{\overline{\mu}}j\overline{g} \end{matrix}\right]\\
&= \left[\begin{matrix} \frac{e^{-\mu}t_1-e^{\overline{\mu}}t_2}{t_1-t_2} & 0 & \frac{t(e^{\overline{\mu}}t_2-e^{-\mu}t_1)+t_1t_2(e^{\overline{\mu}}-e^{-\mu})}{t_1-t_2}i \\ 0 &  e^{\mu-\overline{\mu}} & 0 \\
 \frac{e^{\overline{\mu}}-e^{-\mu}}{t_1-t_2}i & 0 & \frac{t(e^{\overline{\mu}}-e^{-\mu})+e^{\overline{\mu}}t_1-e^{-\mu}t_2}{t_1-t_2} \end{matrix}\right].
\end{align*}
Furthermore, a direct calculation of $\beta$ is
\begin{align*}
\beta &= RE(\nu)R^{-1}\\ &= \left[\begin{matrix} is_1g' & 0 & is_2j' \\ 0 & -i\overline{g'j'}(s_1-s_2) & 0 \\
g' & 0 & j'
\end{matrix} \right] \left[\begin{matrix} e^\nu & 0 & 0 \\ 0 & e^{\overline{\nu}-\nu} & 0 \\
0 & 0 & e^{-\overline{\nu}} \end{matrix}\right] \left[\begin{matrix} \overline{j'} & 0 & -is_2\overline{j'} \\
0 & ig'j'(s_1-s_2) & 0 \\ \overline{g'}
& 0 & -is_1\overline{g'}
\end{matrix}\right]\\
&= \left[\begin{matrix} is_1e^{\nu}g'\overline{j'}+is_2e^{-\overline{\nu}}j'\overline{g'} & 0 & s_1s_2(e^{\nu}g'\overline{j'}+e^{-\overline{\nu}}j'\overline{g'}) \\ 0 &  e^{\overline{\nu}-\nu} & 0 \\
e^{\nu}g'\overline{j'}+e^{-\overline{\nu}}j'\overline{g'} & 0 & -is_2e^{\nu}g'\overline{j'}-is_1e^{-\overline{\nu}}j'\overline{g'} \end{matrix}\right]\\
&= \left[\begin{matrix} \frac{s_1e^{\nu}-s_2e^{-\overline{\nu}}}{s_1-s_2} & 0 & \frac{e^{-\overline{\nu}}-e^{\nu}}{s_1-s_2}s_1s_2i \\ 0 &  e^{\overline{\nu}-\nu} & 0 \\
 \frac{e^{-\overline{\nu}}-e^{\nu}}{s_1-s_2}i & 0 & \frac{s_1e^{-\overline{\nu}}-s_2e^{\nu}}{s_1-s_2} \end{matrix}\right].
\end{align*}
By comparing entries of $\beta$, we get the following five
identities.
\begin{enumerate}
\item  $\displaystyle{\frac{e^{-\mu}t_1-e^{\overline{\mu}}t_2}{t_1-t_2}= \frac{s_1e^{\nu}-s_2e^{-\overline{\nu}}}{s_1-s_2}}$
\item  $\displaystyle{\frac{e^{\overline{\mu}}-e^{-\mu}}{t_1-t_2}= \frac{e^{-\overline{\nu}}-e^{\nu}}{s_1-s_2}}$
\item  $\displaystyle{e^{\mu-\overline{\mu}}=e^{\overline{\nu}-\nu}}$
\item  $\displaystyle{\frac{t(e^{\overline{\mu}}t_2-e^{-\mu}t_1)+t_1t_2(e^{\overline{\mu}}-e^{-\mu})}{t_1-t_2}=\frac{e^{-\overline{\nu}}-e^{\nu}}{s_1-s_2}s_1s_2}$
\item  $\displaystyle{\frac{t(e^{\overline{\mu}}-e^{-\mu})+e^{\overline{\mu}}t_1-e^{-\mu}t_2}{t_1-t_2}=\frac{s_1e^{-\overline{\nu}}-s_2e^{\nu}}{s_1-s_2}}$
\end{enumerate}
By adding the first equation and the last one, we get
$$
\frac{t(e^{\overline{\mu}}-e^{-\mu})}{t_1-t_2}+e^{\overline{\mu}}+e^{-\mu}=e^{\nu}+e^{-\overline{\nu}}.
$$
Hence,
$$
t_1-t_2=\frac{(e^{\overline{\mu}}-e^{-\mu})t}{e^{\nu}+e^{-\overline{\nu}}-e^{\overline{\mu}}-e^{-\mu}}.
$$
If we substitute it to the second equation,
$$
s_1-s_2=\frac{e^{-\overline{\nu}}-e^{\nu}}{e^{\overline{\mu}}-e^{-\mu}}(t_1-t_2)
=\frac{e^{-\overline{\nu}}-e^{\nu}}{e^{\overline{\mu}}-e^{-\mu}}
\cdot
\frac{(e^{\overline{\mu}}-e^{-\mu})t}{e^{\nu}+e^{-\overline{\nu}}-e^{\overline{\mu}}-e^{-\mu}}
=\frac{(e^{-\overline{\nu}}-e^{\nu})t}{e^{\nu}+e^{-\overline{\nu}}-e^{\overline{\mu}}-e^{-\mu}}.
$$
Now we subtract (1) from (5), then
$$
\frac{(t+t_1+t_2)(e^{\overline{\mu}}-e^{-\mu})}{t_1-t_2}=\frac{(s_1+s_2)(e^{-\overline{\nu}}-e^{\nu})}{s_1-s_2}.
$$
Here, we substitute $t_1-t_2$ and $s_1-s_2$, then we get
$t+t_1+t_2=s_1+s_2$. Hence
\begin{align*}
t &= s_1+s_2-t_1-t_2\\
&=
s_1+s_1-\frac{(e^{-\overline{\nu}}-e^{\nu})t}{e^{\nu}+e^{-\overline{\nu}}-e^{\overline{\mu}}-e^{-\mu}}-
\frac{(e^{\overline{\mu}}-e^{-\mu})t}{e^{\nu}+e^{-\overline{\nu}}-e^{\overline{\mu}}-e^{-\mu}}-t_2-t_2,
\end{align*}

so
$$
(e^{\nu}+e^{-\overline{\nu}}-e^{\overline{\mu}}-e^{-\mu}+e^{-\overline{\nu}}-e^{\nu}+e^{\overline{\mu}}-e^{-\mu})t
=2(e^{\nu}+e^{-\overline{\nu}}-e^{\overline{\mu}}-e^{-\mu})(s_1-t_2)
\quad (*)
$$
Now let's calculate the cusp gap function\\
$$
\displaystyle{W(P)=W_\alpha(\gamma^-, \beta^+)=\frac{\partial_y \log[\alpha^+,\gamma^-,y,\beta^+]}{\partial_y \log[\alpha^+,s_0,y,\alpha(s_0)]} \mid_{y=\alpha^+}}.
$$

Here $\alpha^+=\left(\begin{matrix} 1 \\ 0 \\ 0
\end{matrix}\right)$, $\gamma^-=\left(\begin{matrix} it_2 \\ 0 \\
1 \end{matrix}\right)$, and $\beta^+=\left(\begin{matrix} is_1 \\
0 \\ 1 \end{matrix}\right)$. Since $W(P)$ is independent of $s_0$,
we substitute $s_0=\gamma^-$ and since $y$ is a point on boundary,
we let $y=\left(\begin{matrix} ix \\ 0 \\ 1 \end{matrix}\right)$,
where $x$ is a real number, then $y=\alpha^+$ means $x=\infty$.

\begin{align*}
& W(P)  =W_\alpha(\gamma^-, \beta^+)=\frac{\partial_y
\log[\alpha^+,\gamma^-,y,\beta^+]}{\partial_y
\log[\alpha^+,s_0,y,\alpha(s_0)]} \mid_{y=\alpha^+}\\
&= \frac{\partial_y \log \frac{<\alpha^+, \gamma^-><y,
\beta^+>}{<\alpha^+, \beta^+><y, \gamma^->}}{\partial_y \log
\frac{<\alpha^+, s_0><y, \alpha(s_0)>}{<\alpha^+, \alpha(s_0)><y,
s_0>}} \mid_{y=\alpha^+}\\
&= \frac{\partial_y[\log <\alpha^+, \gamma^->+\log <y,
\beta^+>-\log <\alpha^+, \beta^+>-\log <y,
\gamma^->]}{\partial_y[\log <\alpha^+, \gamma^->+\log <y,
\alpha(\gamma^-)>-\log <\alpha^+, \alpha(\gamma^-)>-\log <y, \gamma^->]} \mid_{y=\alpha^+}\\
&= \frac{\partial_x[\log \left<\left(\begin{matrix} ix \\ 0 \\ 1
\end{matrix}\right),\left(\begin{matrix} is_1 \\ 0 \\ 1
\end{matrix}\right)\right>-\log \left<\left(\begin{matrix} ix \\ 0 \\ 1
\end{matrix}\right),\left(\begin{matrix} it_2 \\ 0 \\ 1
\end{matrix}\right)\right>]}{\partial_x[\log \left<\left(\begin{matrix} ix \\ 0 \\ 1
\end{matrix}\right),\left(\begin{matrix} i(t+t_2) \\ 0 \\ 1
\end{matrix}\right)\right>-\log \left<\left(\begin{matrix} ix \\ 0 \\ 1
\end{matrix}\right),\left(\begin{matrix} it_2 \\ 0 \\ 1
\end{matrix}\right)\right>]} \mid_{x=\infty}\\
&= \frac{\partial_x[\log(ix-is_1)-\log(ix-it_2)]}{\partial_x[\log(ix-i(t+t_2))-\log(ix-it_2)]} \mid_{x=\infty}\\
&= \frac{\frac{1}{x-s_1}-\frac{1}{x-t_2}}{\frac{1}{x-(t+t_2)}-\frac{1}{x-t_2}} \mid_{x=\infty}\\
&= \frac{\frac{s_1-t_2}{(x-s_1)(x-t_2)}}{\frac{t}{(x-t-t_2)(x-t_2)}} \mid_{x=\infty}\\
&= \frac{s_1-t_2}{t}.
\end{align*}
By the equation (*),
\begin{align*}
W(P) & =
\frac{s_1-t_2}{t}=\frac{e^{-\overline{\nu}}-e^{-\mu}}{e^{\nu}+e^{-\overline{\nu}}-e^{\overline{\mu}}-e^{-\mu}}
=\frac{1}{\frac{e^{\nu}+e^{-\overline{\nu}}-e^{\overline{\mu}}-e^{-\mu}}{e^{-\overline{\nu}}-e^{-\mu}}}=
\frac{1}{1+\frac{e^{\nu}-e^{\overline{\mu}}}{e^{-\overline{\nu}}-e^{-\mu}}}\\
&=
\frac{1}{1+\frac{e^{\nu}-e^{\mu+\nu-\overline{\nu}}}{e^{-\overline{\nu}}-e^{-\mu}}}
\quad (\because e^{\overline{\mu}}=e^{\mu+\nu-\overline{\nu}})\\
&=
\frac{1}{1+\frac{e^{\mu+\nu}(e^{-\mu}-e^{-\overline{\nu}})}{e^{-\overline{\nu}}-e^{-\mu}}}\\
&=
\frac{1}{1-e^{\mu+\nu}}=\frac{1}{1-e^{\frac{\mu+\overline{\mu}+\nu+\overline{\nu}}{2}}}\\
&= \frac{1}{1-e^{\frac{\mu+\overline{\mu}}{2}}e^{\frac{\nu+\overline{\nu}}{2}}}\\
&= \frac{1}{1-e^{\frac{\ell(\beta)+2n\pi i}{2}}e^{\frac{\ell(\gamma)+2m\pi i}{2}}},
\end{align*}
where $m$ and $n$ are integers. Here, we note that for $A_1, A_2, A_3 \in SL(2,\bc)$,
$\displaystyle{tr(A_i)=\pm2\cosh(\frac{\ell(A_i)}{2})}$ and $\displaystyle{tr(A_1)tr(A_2)tr(A_3)<0}$.\\
Hence, among $tr(A_i)$ for $i=1,2,3$, either only one has negative value or all have negative values.\\
By similar argument, since $\ell(\alpha)=0$ in our case, either $\cosh(\frac{\ell(\beta)}{2})$ or  $\cosh(\frac{\ell(\gamma)}{2})$ must be negative. Hence, among $m$ and $n$, one is even and the other is odd, so
$$W(P)=\frac{1}{1-e^{\frac{\ell(\beta)+2n\pi i}{2}}e^{\frac{\ell(\gamma)+2m\pi i}{2}}}=\frac{1}{1+e^{\frac{\ell(\beta)+\ell(\gamma)}{2}}}.$$

Similarly, we can show that
\begin{align*}
W^r(P) & = W_{\alpha}(\beta^+,\beta^-)=\frac{s_2-s_1}{t}\\
&= \frac{e^{\nu}-e^{-\overline{\nu}}}{e^{\nu}+e^{-\overline{\nu}}-e^{\overline{\mu}}-e^{-\mu}}\\
&= \frac{\sinh\frac{\ell(\beta)}{2}}{\cosh\frac{\ell(\gamma)}{2}+\cosh\frac{\ell(\beta)}{2}}.
\end{align*}

%
%
\end{document}